\documentclass[11pt]{amsart}
\usepackage[pdftex]{graphicx}
\usepackage{color}
\usepackage[T1]{fontenc}
\usepackage{lmodern}
\usepackage[english]{babel}
\usepackage[utf8]{inputenc}

\usepackage[centering,margin=2.5cm]{geometry}
\usepackage{upgreek}
\usepackage{amsfonts}
\usepackage{verbatim}
\usepackage{etex}
\usepackage{amsmath}
\usepackage[all, cmtip]{xy}
\usepackage{tikz-cd}
\usepackage{upgreek}
\usepackage{amssymb}
\usepackage{amsthm}
\usepackage{extpfeil}
\usepackage{bbm}
\usepackage{paralist}
\usepackage[T1]{fontenc}
\usepackage[colorlinks,citecolor=magenta,pagebackref=true,urlcolor=magenta,
pdftex]{hyperref}

\usepackage{amscd}

\usepackage{nicefrac}
\usepackage{microtype}

\usepackage{mathrsfs}
\usepackage[font=small,labelfont=bf]{caption}
\usepackage[labelformat=simple]{subcaption}

\makeatletter
\newtheorem*{rep@theorem}{\rep@title}
\newcommand{\newreptheorem}[2]{%
\newenvironment{rep#1}[1]{%
 \def\rep@title{#2~\ref{##1}}%
 \begin{rep@theorem}}%
 {\end{rep@theorem}}}

\makeatother

\theoremstyle{plain}
\newtheorem*{thm*}{Theorem}
\newtheorem{thm}{Theorem}

\newreptheorem{mthm}{Main Theorem}
\newreptheorem{mcor}{Corollary}
\newreptheorem{mlem}{Lemma}
\newreptheorem{thm}{Theorem}

\newtheorem{cor}[thm]{Corollary}

\newtheorem*{lem*}{Lemma}
\newtheorem{prp}[thm]{Proposition}

\theoremstyle{definition}

 \newcommand{\R}{\mathbb{R}}

\newcommand{\CC}{\mathbb{C}}

\newcommand{\PP}{\mathscr{P}}

\newcommand{\RS}{\mathcal{R}}

\newcommand{\comp}{\mbf{c}}
\newcommand{\vanish}[1]{}

\def\({\left(}
\def\){\right)}
\def\no={\,{\,|\!\!\!\!\!=\,\,}}

\def\no={\,{\,|\!\!\!\!\!=\,\,}}

\newcommand{\xqedhere}[2]{%
  \rlap{\hbox to#1{\hfil\llap{\ensuremath{#2}}}}}

\newcommand{\cm}[1]{}
\newcommand\mc[1]{\mathcal{#1}}

\newcommand\mbf[1]{\mathbf{#1}}

\newcommand\mr[1]{\mathrm{#1}}

\newcommand{\HA}{\mathcal{A}}

\newcommand\RR{\mathbb{R}}

\newcommand{\s}{\mathbf{s}}

\begin{document}
\title[Boundary manifolds and Lefschetz Theorems]{A note on boundary manifolds of arrangements}
\author{Karim A.~Adiprasito}
\address{School of Mathematics, Institute for Advanced Study, Princeton, US
and \newline
Einstein Institute of Mathematics, University of Jerusalem, Jerusalem, Israel}
\email{adiprasito@math.fu-berlin.de, adiprasito@ias.edu}
\date{\today}
\thanks{K.~A.~Adiprasito acknowledges support by an IPDE/EPDI postdoctoral fellowship, a Minerva postdoctoral fellowship of the Max Planck Society, and NSF Grant DMS 1128155.
}

\begin{abstract}
We note an intimate connection between the Lefschetz Theorem for 
$c$-arrangements, and a theorem of Hironaka relating the complement 
of an arrangement to its boundary manifold. This results in a 
generalization of Hironaka's result.
\end{abstract}

\maketitle

If $\HA$ is any subspace arrangement in $\RR^d$ or $\CC^d$, then the boundary manifold $\mr{B}(\HA)$ of $\HA$ is defined as the boundary of $\mbf{N}(|\HA|)$, the regular neighborhood of the underlying variety of $\HA$. The boundary manifold forms an important ingredient in the study of complements of arrangements: it is well-behaved enough to allow for a general study if $d$ is small, and especially fruitful for arrangements of lines in $\CC^2$ resp.\ $\mathbb{CP}^2$ : In this case, it follows by classical results that
\begin{compactitem}[$\circ$]
\item  the boundary manifold is not only built by pieces along the arrangement, it is built by pieces that fiber over the circle, and hence is a graphmanifold in the sense of Waldhausen.
\item Finally, $\mr{B}(\HA)$ admits a nonpositive length metric and is therefore aspherical (unless it is a pencil)~\cite{Schroeder}. Note that this is not true for the complement of complex hyperplane arrangements.
\end{compactitem}
Finally, the boundary manifold is combinatorial:
\begin{prp}[cf.\ \cite{CS}]
The boundary manifold of a complex hyperplane arrangement $\HA$ in $\CC^2$ is determined, up to homotopy equivalence, by the order complex of its intersection poset.
\end{prp}

\begin{proof}
Observe that the homotopy type of $\mr{B}(\HA)$ is obtained by gluing
\begin{compactitem}[$\circ$]
\item tubular pieces (around the line strata of $\HA$) and
\item copies of $S^2\setminus \{\text{open disks}\} \times S^1$ at the vertices (because the vertex links are Hopf linked great circles). 
\end{compactitem}
Since the attaching maps and the local pieces are combinatorially determined, so is the homotopy type of the glued manifold. 
\end{proof}

Note that this does not work for complex hyperplane arrangements in $\CC^3$ by Rybnikov's example \cite{Rybnikov} and $2$-arrangements (following for instance a classical example of Ziegler).

A central motivation for the study of boundary manifolds of arrangements is provided by the study of arrangement complements initiated by Arnold, Brieskorn and others. Hironaka, in an attempt to study complements of complexified real arrangements, provided a beautiful result stating that the complement $\HA^{\comp}$ of the arrangement can be reobtained in a rather simple fashion that is hard to surpass in beauty. Recall: a poset is canonically interpreted topologically via the order complex.

\begin{thm}[Hironaka, cf.\ \cite{Hironaka}]
If $\HA$ is complexified real in $\CC^2$, then the complement $\HA^{\comp}$ is obtained as a quotient of $\mr{B}(\HA)$ over an embedding of the intersection poset $\mc{P}(\HA)$ of $\HA$.
\end{thm}

Hironaka's original theorem has several stronger aspects that we cannot recover in our more general setting; for instance, it provides a detailed descriplition of the attaching map that depends on orientation data that is not as easy to recover. 

We provide a vast generalization and refinement of Hironaka's theorem. The crucial step by exploiting a duality between the Lefschetz theorem for complements (in its full generality proven in \cite{A}) and the Hironaka theorem. This paper is a close follow-up to \cite{A}, and we refer to the relevant notation there.

\section{The Lefschetz theorem for central $c$-arrangements}

Let us recall the following main lemma of \cite{A}.

\begin{thm}[Theorem 3.2, \cite{A}]\label{thm:hemisphere}
Let $\HA$ be a nonempty $c$-arrangement in projective general position in $\RR^d$, and let $\s$ denote a combinatorial stratification of $(\RR^d,|\HA|)$. Then, for any $k$-dimensional subspace $H$ of $\s$ containing an element of $\HA$, the pair $(\RS(\s,H), \RS(\s,|\HA|\cap H))$ is out-$\iota_c(d)$ collapsible.
\end{thm}

Here $\iota_c(d) := \left\lfloor \nicefrac{d}{c} \right\rfloor  - 1$. 
Moreover, a \emph{collapsible pair} $(\Delta,\Gamma)$ is a relative CW complex such that a collapse of $\Delta$ restricts to a collapse of $\Gamma$; similarly, an \emph{out-$i$ collapsible} pair is given if a collapse of $\Delta$ restricts to a Morse funcion on $\Gamma$ whose critical indices are $i$. $\RS(\Delta,A)$ is the restriction of a CW complex to the cells supported in a set $A$.

\begin{proof}[Idea of proof] The proof is simple, and follows an induction procedure. If $H\in \HA$, the theorem is clear. Otherwise, consider a codimension one subspace $h$ of $H$ containing an element of $\HA$, and prove that $(\RS(\s,H), \RS(\s,|\HA|\cap H))$ out-$\iota_c(d)$ collapses to  $(\RS(\s,h), \RS(\s,|\HA|\cap h))$.
\end{proof}

The relevance of this theorem is explained by the fact that it forms the dual to the Morse index estimate of the Lefschetz Hyperplane theorem for complements of complex hyperplane arrangements by Dimca--Papadima \cite{DimcaPapadima}.
More generally, the previous theorem is essentially a dual version of the Lefschetz theorem for central $c$-arrangements. 

\begin{cor}[cf.\ \cite{A}]
Let $\HA$ be a central $c$-arrangement in $\R^d$, and let $H$ denote a general position central hyperplane in $\R^d$. Then $\HA^{\comp}$ is obtained from $\HA^{\comp}\cap H$ by attaching cells of dimension $d-\iota_c(d)+1$.
\end{cor}

\begin{proof}
Pass to the Alexander dual of the Morse function given by Theorem \ref{thm:hemisphere}. 
\end{proof}

In addition to this important fact, Theorem \ref{thm:hemisphere} points to a generalization of Hironaka's theorem to higher dimensions. 

\section{Hironaka's theorem: the complement and the boundary manifold}

\begin{thm}[The general Hironaka Theorem]
Let $\HA$ denote a $c$-arrangement in $\RR^d$, $c>1$. 
Then $\HA^{\comp}$ is a quotient of $\mr{B}(\HA)$ over an embedding $\varphi:\mc{P}(\HA) \hookrightarrow \mr{B}(\HA)$.
\end{thm}

\begin{proof} The case of arrangements in projectively general position is immediate from the Morse function constructed \ref{thm:hemisphere}, as the perfect Morse function for \[(\RS(\s,H), \RS(\s,|\HA|\cap H))\ \cong\ (\R^d,|\HA|)\] descends to a perfect Morse function on $(\R^d,\mbf{N}(|\HA|))$. 

Now, the attaching spheres of critical cells in the relative Morse function can be deformed to an embedding of $\PP(\HA)$ for every $c>1$ by an induction on the dimension: The analysis of critical points of the Morse function on $\HA$ in $\R^d$ is reduced to the study of c-arrangements in $\R^{d-1}$. 

This reduces our reduces our inverstigation to the case of arrangements of points $\R^c$; this is a triviality as long as $c>1$.

For general arrangements, we apply a generic projective transformation $\tau_1$ of $\HA$, and consider a generic homotopy $\tau_t$ of $\tau_1$ to the identity $\tau_0=\mr{id}$, and analyze the change in the boundary manifold and complement every time a singularity is moved to infinity, which can again be done by induction as singularities move to the hyperplane at infinity.
\end{proof}

\section{Two refinements}

While not all properties of complexified line arrangements carry over, 
we can note two immediate refinements of the generalized Hironaka theorem that concern properties of the attaching map $\varphi$ and that still apply. The first concerns the naturality of the embedding~$\varphi$:

\begin{prp}
The embedding $\varphi:\mc{P}(\HA) \hookrightarrow \mr{B}(\HA)$ commutes with the natural homotopy equivalence $\varepsilon:\mc{P}(\HA)\hookrightarrow |\HA|$, i.e., if $\mr{p}:\mr{B}(\HA)\rightarrow |\HA|$ gives a commutative diagram
\[\begin{tikzcd}
\mc{P}(\HA) \arrow[hookrightarrow]{r}{\varepsilon } \arrow[hookrightarrow]{d}{ \varphi } &  \mid\HA\mid \\
\mr{B}(\HA)  \arrow[twoheadrightarrow]{ur}{\mr{p}}& \mbox{\ }
\end{tikzcd}
\]
\end{prp}

The second property concerns a refinement of the attaching properties:

\begin{prp}
Let $\HA$ denote a $c$-arrangement in $\RR^d$, $c>1$. Then, if $\HA$ is non-essential, and $H$ is a general position hyperplane, then the quotient $\mr{B}(\HA)\rightarrow \HA^{\comp}$ commutes with restriction to $H$. In particular, the embedding $\varphi:\mc{P}(\HA)\rightarrow
\mr{B}(\HA)$ can be restricted to $\mr{B}(\HA\cap H)$.
\end{prp}

Both propositions follow immediately from the inductive structure of the proof. In the special case of complex hyperplane arrangements, we in particular obtain:
\begin{cor}
Let $\HA$ denote a complex hyperplane arrangement in $\CC^2$. Then, for every vertex $v$ of $\HA$, the attaching map $\varphi$ can be restricted to a single sheat in $\mr{B}_v(\HA):=\mbf{N}(\mr{p}^{-1}(v))\subset \mr{B}(\HA)$, i.e., 
\[\varphi(\mc{P})\cap \mr{B}_v(\HA)\subset \{a\}\times S^2 \setminus \{\text{open disks}\}\subset \mr{B}_v(\HA) \cong S^1\times S^2 \setminus \{\text{open disks}\}.\]
\end{cor}

\subsection*{Acknowledgements} We wish to thank Alex Suciu for inspiring this note.

{\small
\bibliographystyle{myamsalpha}
\bibliography{MinC-Arr}}

\end{document}